\documentclass[reqno]{amsart}

\usepackage{mystyle}

\usepackage{etoolbox}

\usepackage[headings]{fullpage}
\usepackage{parskip}

\usepackage[utf8]{inputenc}
\usepackage[T1]{fontenc}
\usepackage[british]{babel}
\usepackage[pdfencoding=auto]{hyperref}

\usepackage{microtype}
\usepackage{booktabs}

\usepackage{amsfonts,amssymb,bbm,mathrsfs,stmaryrd}
\usepackage{amsmath}

\usepackage{mathtools}
\usepackage[noabbrev,capitalize]{cleveref}
\usepackage{centernot}
\usepackage[shortlabels]{enumitem}
\usepackage{url}
\usepackage{tikz}
\usepackage{pict2e}
\usepackage{framed}

\usetikzlibrary{matrix}
\usetikzlibrary{positioning}
\usetikzlibrary{arrows}
\tikzset{> =stealth}
\usetikzlibrary{arrows.meta}
\tikzset{normalHead/.tip={Triangle[open,angle=60:4pt]},}
\tikzset{normalTail/.tip={Triangle[reversed,open,angle=60:4pt]},}

\newcommand{\addQEDstyle}[2]{\AtBeginEnvironment{#1}{\pushQED{\qed}\renewcommand{\qedsymbol}{#2}}\AtEndEnvironment{#1}{\popQED}}

\theoremstyle{plain}
\newtheorem{theorem}{Theorem}[section]
\newtheorem{lemma}[theorem]{Lemma}
\newtheorem{proposition}[theorem]{Proposition}

\theoremstyle{definition}
\newtheorem{definition}[theorem]{Definition}

\addQEDstyle{example}{$\triangle$}
\usepackage{caption}

\theoremstyle{remark}
\newtheorem{remark}[theorem]{Remark}

\renewcommand{\epsilon}{\varepsilon}
\renewcommand{\phi}{\varphi}

\newcommand{\N}{\mathbb{N}}

\newcommand{\inv}{^{-1}}

\DeclareMathOperator{\Gl}{Gl}

\newcommand{\splitext}[6]{%
\tikz[baseline]{
\newdimen{\mylabelwidth}
\newdimen{\skipwidth}
\node[anchor=base] (A) {\hspace*{\dimexpr0.5pt-\pgfkeysvalueof{/pgf/inner xsep}}${#1}$};
\settowidth{\mylabelwidth}{\pgfinterruptpicture {$#2$} \endpgfinterruptpicture}
\pgfmathsetlength{\skipwidth}{max(\mylabelwidth,10pt)}
;\node[right] (B) at ([xshift=\skipwidth+12pt]A.east) {${#3}$};
\settowidth{\mylabelwidth}{\pgfinterruptpicture {$#4$} \endpgfinterruptpicture}
\settowidth{\skipwidth}{\pgfinterruptpicture {$#5$} \endpgfinterruptpicture}
\pgfmathsetlength{\skipwidth}{max(\skipwidth,\mylabelwidth,10pt)}
\node[right] (C) at ([xshift=\skipwidth+12pt]B.east) {${#6}$\hspace*{\dimexpr0.5pt-\pgfkeysvalueof{/pgf/inner xsep}}};
\draw[normalTail->] (A) to node [above] {${#2}$} (B);
\draw[transform canvas={yshift=0.5ex},-normalHead] (B) to node [above] {${#4}$} (C);
\draw[transform canvas={yshift=-0.5ex},->] (C) to node [below] {${#5}$} (B);
}}

\newcommand{\normalext}[5]{%
\tikz[baseline]{
\newdimen{\mylabelwidth}
\newdimen{\skipwidth}
\node[anchor=base] (A) {\hspace*{\dimexpr0.5pt-\pgfkeysvalueof{/pgf/inner xsep}}${#1}$};
\settowidth{\mylabelwidth}{\pgfinterruptpicture {$#2$} \endpgfinterruptpicture}
\pgfmathsetlength{\skipwidth}{max(\mylabelwidth,12pt)}
\node[right] (B) at ([xshift=\skipwidth+10pt]A.east) {${#3}$};
\settowidth{\mylabelwidth}{\pgfinterruptpicture {$#4$} \endpgfinterruptpicture}
\pgfmathsetlength{\skipwidth}{max(\mylabelwidth,10pt)}
\node[right] (C) at ([xshift=\skipwidth+10pt]B.east) {${#5}$\hspace*{\dimexpr0.5pt-\pgfkeysvalueof{/pgf/inner xsep}}};
\draw[normalTail->] (A) to node [above] {${#2}$} (B);
\draw[-normalHead] (B) to node [above] {${#4}$} (C);
}}

\usetikzlibrary{shapes.geometric, arrows}

\tikzstyle{startstop} = [rectangle, rounded corners, minimum width=3cm, minimum height=1cm,text centered, text width=4cm, draw=black, fill=white]
\tikzstyle{io} = [trapezium, trapezium left angle=70, trapezium right angle=110, minimum width=3cm, minimum height=1cm, text centered, draw=black, fill=blue!30]
\tikzstyle{process} = [rectangle, minimum width=3cm, minimum height=1cm, text centered, text width=3cm, draw=black, fill=orange!30]
\tikzstyle{decision} = [diamond, minimum width=3cm, minimum height=1cm, text centered, draw=black, fill=green!30]
\tikzstyle{arrow} = [thick,->,>=stealth]

\pgfdeclarelayer{bg}
\pgfdeclarelayer{over}
\pgfsetlayers{bg,main,over}

\tikzset{dot/.style={circle,draw=black,fill=black,minimum size=1mm,inner sep=0mm}}

\makeatletter
\pgfkeys{%
  /tikz/on layer/.code={
    \pgfonlayer{#1}\begingroup
    \aftergroup\endpgfonlayer
    \aftergroup\endgroup
  },
  /tikz/node on layer/.code={
    \gdef\node@@on@layer{%
      \setbox\tikz@tempbox=\hbox\bgroup\pgfonlayer{#1}\unhbox\tikz@tempbox\endpgfonlayer\egroup}
    \aftergroup\node@on@layer
  },
  /tikz/end node on layer/.code={
    \endpgfonlayer\endgroup\endgroup
  }
}
\def\node@on@layer{\aftergroup\node@@on@layer}
\makeatother

\title{A Survey of Schreier-Type Extensions of Monoids}

\author[P. F. Faul]{Peter F. Faul}
\address{Department of Pure Mathematics and Statistical Sciences\\ University of Cambridge}
\email{peter@faul.io}

\date{\today}

\subjclass[2010]{20M50, 18G50, 20M18}
\keywords{semigroup, weakly Schreier, short five lemma, monoid cohomology, factorisable monoids}

\begin{document}

\maketitle

\begin{abstract}
We give an overview of a number of Schreier-type extensions of monoids and discuss the relation between them. We begin by discussing the characterisations of split extensions of groups, extensions of groups with abelian kernel and finally non-abelian group extensions. We see how these characterisations may be immediately lifted to Schreier split extensions, special Schreier extensions and Schreier extensions respectively. Finally, we look at weakenings of these Schreier extensions and provide a unified account of their characterisation in terms of relaxed actions.
\end{abstract}


\section{Introduction}
Monoid extension theory, in many ways, has a different character to that of group extension theory. In group extension theory there exists an obvious and well-behaved class of objects under study, the extensions of groups. Things are more difficult in the setting of monoids. Monoid extensions are far less well-behaved than their group cousins and so to make progress one must restrict to a particular subclass of monoid extensions which are sufficiently group-like. There is no canonical subclass and so the landscape is somewhat fractured. Not only are the connections between these subclasses not always well understood but some promising approaches seem to have been missed or forgotten entirely. 

It is my hope to rectify this in some small way. This survey restricts itself to those classes of extensions which are \emph{Schreier-like}, which informally we may think of as those extensions similar in character to the \emph{Schreier extensions of monoids} introduced by Redei in \cite{redei1952verallgemeinerung}. We will see, in fact, that these constitute a not insignificant portion of extensions in the literature. Some notable classes of extensions which do not fall neatly into this category will also be discussed, including Leech's \emph{$\mathcal{H}$-coextensions} \cite{leech1975two} and Grillet's \emph{left coset extensions} \cite{grillet1974left}.

Many papers that are not strictly about monoid extensions will be encountered too. Much of the work on \emph{quasi-decompositions} of monoids \cite{schmidt1977quasi,krishnan1979quasi,kohler1979quasi} may be reinterpreted in terms of Schreier-type extensions (not that this connection was necessarily lost on the authors.)
We will also see an interesting connection to \emph{magnifying elements} which may imply some connection to \emph{factorisations of monoids} \cite{tolo1969factorizable, gutan1997semigroups}.

\begin{remark}
I would like to briefly draw attention to K\"ohler's largely forgotten paper mentioned above \cite{kohler1979quasi}. His characterisation of quasi-decompositions of monoids is identical to my characterisation of weakly Schreier split extensions in \cite{faul2019characterisation}. Indeed, it is not hard to see that quasi-decompositions and weakly Schreier split extensions are precisely the same thing. It was this discovery that made me see the necessity of writing this survey. 
\end{remark}

\subsection{Outline}
We begin with a background on group extensions and discuss the well-known characterisations of split extensions, extensions with abelian kernel and finally non-abelian group extensions, with attention given to the low-dimensional cohomology groups as appropriate. 

We then discuss the relatively recent work on those extensions most directly related to Redei's Schreier extensions. Specifically we will look at the characterisations of Schreier split extension \cite{martins2013semidirect}, special Schreier extensions \cite{martins2016baer} and finally Schreier extensions themselves \cite{redei1952verallgemeinerung}, as before links to the low-dimensional cohomology groups will be discussed.

Finally we look at a number of extensions which constitute a weakening of those above, doing away with a uniqueness condition in each case. We will give the characterisation of weakly Schreier split extensions \cite{faul2019characterisation} and discuss its link to the work on quasi-decompositions \cite{schmidt1977quasi,krishnan1979quasi,kohler1979quasi}. We then discuss the special weakly Schreier extensions and their characterisation in the abelian kernel setting \cite{faul2020baer}. 
Finally we look at weakly Schreier extensions, first studied and characterised in \cite{fleischer1981monoid} (another forgotten paper). In particular, we give a new presentation of this latter characterisation in terms of relaxed actions of monoids, providing a unified account for all three characterisation closely resembling the group case. It is here that the connections to magnifying elements may be found \cite{migliorini1971some}. 
\subsection{The web of extensions}
Below is a visual depiction of the relationships between the various classes of extensions. Extensions below the dashed line are split. Black arrows indicate a proper inclusion, red arrows a proper inclusion when restricted to group kernel (forgetting the splitting). Not depicted in this diagram is that (weakly) Schreier split extensions are included in (weakly) Schreier extensions.

\begin{center}
\vspace{0.5cm}
\scalebox{0.84}{
\begin{tikzpicture}[node distance=3cm]

\node (A) [startstop] {Weakly Schreier extensions};
\node (B) [startstop, below left of=A] {Special weakly Schreier extensions};
\node (C) [startstop, right =1cm of B] {Schreier extensions};
\node (F) [startstop, below =7cm of B] {Weakly Schreier split extensions};
\node (G) [startstop, left =2cm of F] {Quasi-decompositions of monoids};
\node (H) [startstop, below =1.5cm of G] {Krishnan's quasi-resolutions of monoids};
\node (I) [startstop, below =1.5cm of H] {Quasi-decompositions of commutative monoids};

\node (J) [startstop, below =1.5cm of C] {Special Schreier extensions};
\node (K) [startstop, below =7cm of J] {Schreier split extensions};
\node (D) [startstop, left =4cm of J] {Leech's normal extensions};
\node (E) [startstop, below =1.5cm of D] {Fulp and Stepp's central extensions};

\node (M) [startstop, left =1.9cm of B] {$\mathcal{H}$-coextensions};
\node (i1) [above =1.5cm of D] {};
\node (i2) [above =1.1cm of B] {};
\node (i3) [above =1.1cm of C] {};

\draw [arrow, transform canvas={xshift=0.8cm}] (B) -- (i2);
\draw [arrow,transform canvas={xshift=-1.6cm}] (C) -- (i3);
\draw [arrow,transform canvas={xshift=1.5cm}] (D) -- (i1);
\draw [arrow] (E) -- (D);
\draw [red, arrow] (F) -- (B);
\draw [double equal sign distance] (G) -- (F);
\draw [arrow] (H) -- (G);
\draw [arrow] (I) -- (H);
\draw [arrow] (J) -- (C);
\draw [arrow,red] (K) -- (J);
\draw [arrow] (K) -| (F);
\begin{pgfonlayer}{bg}
\draw [arrow,transform canvas={xshift=1.5cm}] (J) -| (B);
\end{pgfonlayer}
\draw [arrow] (D) -| (M);
\draw[dashed] ([yshift=1cm,xshift=-2cm]G.north) -- ([yshift=1cm,xshift=7cm]F.north);

\end{tikzpicture}}
\end{center}

\section{Background}
\subsection{Split extensions of groups}\label{sec:splitgrp}
Let $N$, $G$ and $H$ be groups. A split extension of groups may be defined as follows.
\begin{definition}
The diagram $\splitext{N}{k}{G}{e}{s}{H}$ is a split extension when $k$ is the kernel of $e$, $e$ the cokernel of $k$ and $s$ a splitting of $e$. 
\end{definition}

It is well known that split extensions may be characterised in terms of group actions via the semidirect product construction. This works in the following way. Each element $g \in G$ may be expressed uniquely as $g = k(n)s(h)$ for $n \in N$ and $h \in H$. It is not hard to verify that the above $h = e(g)$. There exists then a bijection of sets $f\colon N \times H \to G$ sending $(n,h)$ to $k(n)s(h)$. The semidirect product associated with $\splitext{N}{k}{G}{e}{s}{H}$ is the set $N \times H$ equipped with a multiplication making $f$ an isomorphism.

Since $e(f(n_1,h_1)f(n_2,h_2)) = h_1h_2$, we have that $(n_1,h_1) \cdot (n_2,h_2) = (x,h_1h_2)$ for some $x \in N$. Thus we are looking for some $x$ satisfying that $k(x)s(h_1)s(h_2) = f(n_1,h_2)f(n_2,h_2) = k(n_1)s(h_1)k(n_2)s(h_2)$. The idea is to commute $s(h_1)$ past $k(n_2)$. This may be achieved by conjugating $k(n_2)$ by $s(h_1)$. Since $s(h)k(n)s(h)\inv$ belongs to the kernel we get an action $\alpha \colon H \times N \to N$ of $H$ on $N$ satisfying that $k\alpha(h_1,n_2) = s(h_1)k(n_2)s(h_1)\inv$. It then follows immediately that the multiplication may be defined as $(n_1,h_1) \cdot (n_2,h_2) = (n_1 \alpha(h_1,n_2),h_1h_2)$. Call this group $N \rtimes_\alpha H$ the semidirect product and observe that it may be completed into an extension $\splitext{N}{k'}{N \rtimes_\alpha H}{e'}{s'}{H}$, where $k'(n) = (n,1)$, $s'(h) = (1,h)$ and $e'(n,h) = h$.
This split extension will be isomorphic to the original.

In fact starting with groups $H$ and $N$ and an arbitrary action $\alpha\colon H \times N \to N$, one may construct the semidirect product extension as above. Actions then constitute a complete characterisation of split extensions of groups. For more on this see \cite{maclane2012homology}.

\subsection{Group extensions with abelian kernel}\label{sec:abgrp}
The above ideas may be extended to provide a characterisation for group extensions $\normalext{N}{k}{G}{e}{H}$ with abelian kernel $N$. 

\begin{definition}
The diagram $\normalext{N}{k}{G}{e}{H}$ is an extension if $k$ is the kernel of $e$ and $e$ is the cokernel of $k$.
\end{definition}

The main idea is that $e$ is necessarily surjective and so there exist set-theoretic splittings $s$ (which we may always choose to preserve the unit). We would like to emulate the construction above, relative to this splitting $s$, but must find a way to deal with the fact that $s(h_1h_2) \ne s(h_1)s(h_2)$ --- a fact which was used implicitly in the above. 

Again it is possible to express each $g \in G$ uniquely as $g = k(n)se(g)$.
We make use of this fact to prove the following useful lemma.

\begin{lemma}\label{lmm:cosetalcondition}
Let $\normalext{N}{k}{G}{e}{H}$ be an extension of groups. Then if $e(g_1) = e(g_2)$, there exists a unique $n \in N$ such that $g_1 = k(n)g_2$. 
\end{lemma}

\begin{proof}
    Suppose that $e(g_1) = e(g_2)$. We may thus write $g_1 = k(n_1)se(g_1)$ and $g_2 = k(n_2)se(g_2) = k(n_2)se(g_1)$. We then see that $k(n_1n_2\inv)g_2 = k(n_1)se(g_1) = g_1$.
\end{proof}

Now notice that $e(s(h_1)s(h_2)) = e(s(h_1h_2))$ and hence by \cref{lmm:cosetalcondition} there exists a unique $\chi(h_1,h_2) \in N$ such that $s(h_1)s(h_2) = k\chi(h_1,h_2)s(h_1h_2)$. So while it is not the case that $s(h_1)s(h_2) = s(h_1h_2)$, there is a unique element we may multiply $s(h_1h_2)$ on the left by to give the desired $s(h_1)s(h_2)$. 
The argument now carries through as follows. First we define a candidate action $\alpha$ of $H$ on $N$ by $k\alpha(h,n) = s(h)k(n)s(h)\inv$. The only nontrivial condition to check is that $\alpha(h_1h_2,n) = \alpha(h_1,\alpha(h_2,n))$. Here consider the following calculation making explicit use of the fact that $N$ is abelian.  

\begin{align*}
    k\chi(h_1,h_2)k\alpha(h_1h_2,n)s(h_1h_2)   
    &= k\chi(h_1,h_2)s(h_1h_2)k(n) \\
    &= s(h_1)s(h_2)k(n) \\
    &= s(h_1)k\alpha(h_2,n)s(h_2) \\
    &= k\alpha(h_1,\alpha(h_2,n))s(h_1)s(h_2) \\
    &= k\alpha(h_1,\alpha(h_2,n))k\chi(h_1,h_2)s(h_1h_2) \\
    &= k\chi(h_1,h_2)k\alpha(h_1,\alpha(h_2,n))s(h_1h_2)
\end{align*}

Here the last step uses that $N$ is abelian and since each element is invertible we deduce that $\alpha(h_1h_2,n) = \alpha(h_1,\alpha(h_2,n))$ as required.

This action is an invariant of the extension and in particular does not depend on the choice of splitting $s$. If $s'$ was some other splitting then we would have $es(h) = es'(h)$ and so may apply \cref{lmm:cosetalcondition} to acquire a unique element $t(h) \in N$ satisfying that $s(h) = kt(h)s'(h)$. Now if $\alpha$ is the action associated to $s$ we may simply observe the following.

\begin{align*}
    k\alpha(h,n)s'(h)   &= k\alpha(h,n)kt(h)s(h) \\
                        &= kt(h)k\alpha(h,n)s(h) \\
                        &= kt(h)s(h)k(n) \\
                        &= s'(h)k(n)
\end{align*}

Thus the action associated to $s'$ is indeed the same $\alpha$.

Now we may complete the characterisation. Again we have a bijection $f\colon N \times H \to G$ sending $(n,h)$ to $k(n)s(h)$ and would like to equip $N \times H$ with a multiplication making $f$ into an isomorphism. This is given by 
\[
(n_1,h_1)(n_2,h_2) = (n_1\alpha(h_1,n_2)\chi(h_1,h_2),h_1h_2)
\]

Indeed, we may verify that everything is in order.

\begin{align*}
    f((n_1,h_1)(n_2,h_2))   &= f(n_1\alpha(h_1,n_2)\chi(h_1,h_2),h_1h_2) \\
                            &= k(n_1)k\alpha(h_1,n_2)k\chi(h_1,h_2)s(h_1h_2) \\
                            &= k(n_1)k\alpha(h_1,n_2)s(h_1)s(h_2) \\
                            &= k(n_1)s(h_1)k(n_2)s(h_2) \\
                            &= f(n_1,h_1)f(n_2,h_2)
\end{align*}

The function $\chi$ as defined always satisfies that  
\begin{align}
    &\chi(1,h) = 1 = \chi(h,1) \\
    &\chi(x,y)\chi(xy,z) = \alpha(x,\chi(y,z))\chi(x,yz).
\end{align}

Given an action $\alpha\colon H \times N \to N$, any function $\chi\colon H \times H \to N$ satisfying conditions (1) and (2) we call a factor set relative to $\alpha$. Given the data of an action $\alpha$ and a factor set $\chi$, we can equip $N \times H$ with a multiplication as described above. This may always be extended to an extension $\normalext{N}{k}{N \rtimes_\alpha^\chi H}{e}{H}$ in which $k(n) = (n,1)$ and $e(n,h) = h$.

Let us denote by $\mathcal{Z}^2(H,N,\alpha)$ the set of all factor sets relative to the action $\alpha\colon H \times N \to N$. This may be imbued with an abelian group structure given by the pointwise multiplication of the factor sets in $N$. Notice however that we do not have a bijection between $\mathcal{Z}^2(H,N,\alpha)$ and the set of isomorphism classes of extensions with associated action $\alpha$. This is because each splitting of $e$ has associated to it a different factor set. One may identify these factor sets and the equivalence relation is in fact a congruence whose associated normal subgroup is the set of inner-factor sets. These may be derived as follows.

Let $\normalext{N}{k}{G}{e}{H}$ be an extension with abelian kernel and associated action $\alpha$ and suppose that $s$ and $s'$ are two set-theoretic splittings of $e$ with associated factor sets $\chi$ and $\chi'$ respectively. By \cref{lmm:cosetalcondition} we know that there exists a function $t\colon H \to N$ such that $s(h) = kt(h)s'(h)$. Now we may consider the following derivation.

\begin{align*}
    s(h_1)s(h_2) 
    &= kt(h_1)s'(h_1)kt(h_2)s'(h_2) \\
    &= kt(h_1)k\alpha(h_1,t(h_2))s'(h_1)s'(h_2) \\
    &= kt(h_1)k\alpha(h_1,t(h_2))\chi'(h_1,h_2)s'(h_1h_2) \\
    &= kt(h_1)k\alpha(h_1,t(h_2))\chi'(h_1,h_2)kt(h_1h_2)\inv s(h_1h_2)
\end{align*}
 
Since $N$ is abelian, this final line may be reordered to give \[(kt(h_1)k\alpha(h_1,t(h_2))kt(h_1h_2)\inv) \chi'(h_1,h_2)s(h_1h_2) = s(h_1)s(h_2).\] If we write $\delta t(h_1,h_2) = kt(h_1)k\alpha(h_1,t(h_2))kt(h_1h_2)\inv$ we may show that this is a factor set and moreover we have that $\chi(h_1,h_2) = \delta t(h_1,h_2) \chi'(h_1,h_2)$. Factor sets of the form $\delta t$ for $t\colon H \to N$ a unit preserving map, are called inner factor sets.

We call the resulting quotient $\mathcal{H}^2(H,N,\alpha)$ the \emph{second cohomology group} corresponding to $\alpha$. By definition there is an isomorphism from $\mathcal{H}^2(H,N,\alpha)$ onto the set of isomorphism classes of extensions with associated action $\alpha$ and this induces a multiplication on the set known as the \emph{Baer sum}. 

For more information on this construction, see \cite{maclane2012homology}.

\subsection{Characterising non-abelian group extensions}\label{sec:nonabgrp}
We made use of $N$'s commutativity at various points in the above characterisation, most notably when proving that the associated $\alpha$ was indeed an action. It is in fact possible to do away with this assumption and characterise arbitrary group extensions at the cost of $\alpha$ not being an action. The study of this approach is known as Schreier theory \cite{schreier1926erweiterung}.

Given a group extension $\normalext{N}{k}{G}{e}{H}$ we may always define an action of $G$ on $N$ given by conjugation, as $gk(n)g\inv$ will always lie in the kernel. In the case that $N$ is abelian when we precompose the action with an arbitrary splitting $s$ the result is an action of $H$ on $N$. However, in the general non-abelian case, the resulting function $\alpha\colon H \times N \to N$ may fail to satisfy that $\alpha(h_1h_2,n) = \alpha(h_1,\alpha(h_2,n))$.

If $\chi$ is the factor set associated to the splitting $s$ then note that the following identity holds.

\begin{align*}
    k\chi(h_1,h_2)k\alpha(h_1h_2,n)s(h_1h_2) 
    &= k\chi(h_1,h_2)s(h_1h_2)k(n) \\
    &= s(h_1)s(h_2)k(n) \\
    &= s(h_1)k\alpha(h_2,n)s(h_2) \\
    &= k\alpha(h_1,\alpha(h_2,n))s(h_1)s(h_2) \\
    &= k\alpha(h_1,\alpha(h_2,n))k\chi(h_1,h_2)s(h_1h_2)
\end{align*}

Thus, we may conclude that $\chi(h_1,h_2)\alpha(h_1h_2,n) = \alpha(h_1,\alpha(h_2,n))\chi(h_1,h_2)$. This condition is enough to ensure that the desired multiplication on $N \times H$, that being $(n_1,h_1)(n_2,h_2) = (n_1\alpha(h_1,n_2)\chi(h_1,h_2),h_1h_2)$, would indeed be associative. 

Thus, general group extensions may be constructed by pairs $(\alpha, \chi)$ called \emph{factor systems} satisfying the following conditions.

\begin{enumerate}
    \item $\alpha(1,n) = n$,
    \item $\alpha(h,1) = 1$,
    \item $\alpha(h,n_1n_2) = \alpha(h,n_1)\alpha(h,n_2)$,
    \item $\chi(h_1,h_2)\alpha(h_1h_2,n) = \alpha(h_1,\alpha(h_2,n))\chi(h_1,h_2)$,
    \item $\chi(1,h) = 1 = \chi(h,1)$,
    \item $\chi(x,y)\chi(xy,z) = \alpha(x,\chi(y,z))\chi(x,yz)$.
\end{enumerate}

We call the resulting group $N \rtimes_\alpha^\chi H$ the \emph{crossed product} associated to $(\alpha,\chi)$. It may be viewed as an extension in the obvious way.

To provide a full characterisation we must identify when two pairs give the same extensions. Suppose that $(\alpha,\chi)$ and $(\alpha',\chi')$ give isomorphic extensions. 

Suppose we have the following isomorphism of extensions.

\begin{center}
  \begin{tikzpicture}[node distance=2.0cm, auto]
    \node (A) {$N$};
    \node (B) [right=1.2cm of A] {$N \rtimes_{\alpha}^\chi H$};
    \node (C) [right=1.2cm of B] {$H$};
    \node (D) [below of=A] {$N$};
    \node (E) [below of=B] {$N \rtimes_{\alpha'}^{\chi'} H$};
    \node (F) [below of=C] {$H$};
    \draw[normalTail->] (A) to node {$k$} (B);
    \draw[-normalHead] (B) to node {$e$} (C);
    \draw[normalTail->] (D) to node {$k'$} (E);
    \draw[-normalHead] (E) to node {$e'$} (F);
    \draw[->] (B) to node {$f$} (E);
    \draw[double equal sign distance] (A) to (D);
    \draw[double equal sign distance] (C) to (F);
  \end{tikzpicture}
\end{center}

Since $f$ is a morphism of extensions we have that $f(n,1) = fk(n) = k'(n) = (n,1)$, and also that it preserves the second component. Since each element may be written $(n,h) = (n,1)(1,h)$ we have that $f(n,h) = f(n,1)f(1,h) = (n,1)(\gamma(h),h) = (n\gamma(h),h)$ where $\gamma\colon H \to N$ is a set-theoretic map preserving $1$. 

Thus in general $(\alpha,\chi)$ and $(\alpha',\chi')$ will give isomorphic extensions when there exists a map $\gamma\colon H \to N$ such that when we define $f$ as above the result is a homomorphism. (It is easy to see that $f$ is always a bijection of sets.)

By simply computing and comparing $f((n_1,h_1)(n_2,h_2))$ and $f(n_1,h_1)f(n_2,h_2)$ we see that this will occur whenever for all $n\in N$ and $h \in H$ we have 
\begin{align*}
    n_1\alpha'(h_1,n_2)\chi'(h_1,h_2)\gamma(h_1h_2) = n_1\gamma(h_1)\alpha(h_1,\gamma(h_2)n_2)\chi(h_1,h_2).
\end{align*}

The existence of such a $\gamma$ defines an equivalence relation on the set of factor systems which then gives the complete characterisation of group extensions.

\section{Strict Schreier-type extensions of monoids}

In the group setting conjugation appeared to play a crucial role in each characterisation. This poses some problems for the setting of monoids and indeed not much progress can be made when dealing with completely general monoid extensions. However, it is possible to identify a reasonable class of extensions for which all of the above arguments carry through easily. In this section we study the (strict) Schreier extensions of monoids and some of their specialisations. These were introduced by Redei in \cite{redei1952verallgemeinerung} and are defined as follows.

\begin{definition}
An extension $\normalext{N}{k}{G}{e}{H}$ of monoids is \emph{Schreier} if in each fibre $e\inv(h)$ there exists an element $u_h$, called a \emph{generator}, such that for each element $g \in e\inv(h)$, there exists a unique element $n \in N$ such that $g = k(n)u_h$.
\end{definition}

These are the extensions for which there is a natural bijection of sets $f\colon N \times H \to G$, given by $f(n,h) = k(n)u_h$. Even without conjugation, this is enough to generalise all of the preceding constructions.

From this class we may then consider a notion of split extension as well as a class amenable to generalizing the second cohomology groups. With respect to this latter point it should be mentioned that in \cite{tuen1976nonabelianextensions} Schreier extensions of $H$ by an $H$-module were classified in terms of a second cohomology group. These results were then generalised to semi-modules in \cite{patchkoria1977extensions} and \cite{patchkoria1979schreier}. Our focus in this paper is on extensions of monoids by monoids and so these results will not be discussed further.


\subsection{Characterising Schreier split extensions}
Schreier split extensions are defined as follows.

\begin{definition}
A split extension $\splitext{N}{k}{G}{e}{s}{H}$ is a Schreier split extension when for each $g \in G$ there exists a unique $n \in N$ such that $g = k(n)se(g)$.
\end{definition}

Note that these are just Schreier extensions with a distinguished splitting $s$ which selects the generating elements in each fibre.

This condition is strong enough to characterise them in terms of monoid actions in a manner completely analogously to the group setting. Note that $s(h)k(n) \in e\inv(h)$ and so there exists a unique element $\alpha(h,n) \in N$ such that $s(h)k(n) = k\alpha(h,n)s(h)$. These $\alpha(h,n)$ arrange into an action $\alpha\colon H \times N \to N$ which is the analogue of conjugation in the group setting. Indeed, because it is uniquely defined, when the monoids in questions are groups the associated $\alpha$ will of course correspond to conjugation.

Recall that in the characterisation of split extensions of groups it was never important that $\alpha$ was conjugation. What was important was that $k\alpha(h,n)s(h) = s(h)k(n)$. Hence it is fairly easy to see that the same construction works here and that Schreier split extensions are in one-to-one correspondence with monoid actions of $H$ on $N$. 

Concretely, we may construct the semidirect product $N \rtimes_\alpha H$ from an action $\alpha$ of $H$ on $N$. It's underlying set is $N \times H$ and multiplication is given by
\[
(n_1,h_1)(n_2,h_2) = (n_1\alpha(h_1,n_2),h_1h_2).
\]

Naturally, this forms a part of an extension $\splitext{N}{k}{G}{e}{s}{H}$ with $k(n) = (n,1)$, $e(n,h) = h$ and $s(h) = (1,h)$.
This characterisation may be found in \cite{martins2013semidirect} among a number of other results.

These extensions serve as the motivating example for $\mathcal{S}$-protomodular categories \cite{bourn2015Sprotomodular}. Briefly, protomodular categories \cite{bourn1991normalization} are those categories that `behave' like the category of groups. Many nice properties exist in these categories including the `good behaviour' of split extensions. The category of monoids is not protomodular, in particular split extensions in general are `badly behaved'. However a restricted class of split extensions may be considered and a notion of protomodularity may be considered relative to this class. 

\subsection{Characterising Schreier extensions with abelian group kernel}
Schreier extensions with group kernel are known as \emph{special Schreier extensions} and it is not hard to see that they can be defined equivalently as follows.

\begin{definition}
An extension $\normalext{N}{k}{G}{e}{H}$ is special Schreier if whenever $e(g_1) = e(g_2)$ there exists a unique $n \in N$ such that $g_1 = k(n)g_2$.
\end{definition}

Notice that a Schreier split extension will also give rise to special Schreier precisely when it has group kernel.

We see that this definition enforces that \cref{lmm:cosetalcondition} will hold for this class of extensions. Since \cref{lmm:cosetalcondition} and the commutativity of $N$ were the only tools used in the characterisation of group extensions with abelian kernel it makes sense that abelian special Schreier extensions may be characterised just as in the group case.

Indeed, if we choose an arbitrary set-theoretic splitting of $e$ preserving $1$, we may use the special Schreier property to define a unique factor set $\chi\colon H \times H \to N$. Moreover an action $\alpha$ may be defined via the unique elements $\alpha(h,n)$ such that $k\alpha(h,n)s(h) = s(h)k(n)$.

Concretely, when given an action $\alpha$ of $H$ on $N$ and a factor set $\chi$ the crossed product $\N \rtimes_\alpha^\chi H$ may be defined as the set $N \times H$ with multiplication given by 
\[
(n_1,h_1)(n_2,h_2) = (n_1\alpha(h_1,n_2)\chi(h_1,h_2),h_1h_2).
\]
We may then form the extension $\normalext{N}{k}{G}{e}{H}$ in which $k(n) = (n,1)$ and $e(n,h) = h$.

Adapting the arguments in the group case, we may again arrange all the factor sets relative to an action $\alpha$ into an abelian group $\mathcal{Z}^2(H,N,\alpha)$ and may quotient it by the inner factor sets to arrive at the second cohomology group $\mathcal{H}^2(H,N,\alpha)$. This not only provides a full characterisation of special Schreier extensions of monoids with abelian kernel, it supplies them with a Baer sum. 

This characterisation was first done in \cite{martins2016baer} among a number of other results regarding special Schreier extensions in general.

\subsection{Characterising Schreier extensions}
The argument in \cref{sec:nonabgrp} has been setup so that it carries through immediately to the Schreier case. Of course, all Schreier split extensions and special Schreier extensions are Schreier.

Thus we see that Schreier extensions of monoids are characterised by factor systems $(\alpha,\chi)$ where $\alpha\colon H \times N \to N$ and $\chi\colon H \times H \to N$ satisfy the following conditions.

\begin{enumerate}
    \item $\alpha(1,n) = n$,
    \item $\alpha(h,1) = 1$,
    \item $\alpha(h,n_1n_2) = \alpha(h,n_1)\alpha(h,n_2)$,
    \item $\chi(h_1,h_2)\alpha(h_1h_2,n) = \alpha(h_1,\alpha(h_2,n))\chi(h_1,h_2)$,
    \item $\chi(1,h) = 1 = \chi(h,1)$,
    \item $\chi(x,y)\chi(xy,z) = \alpha(x,\chi(y,z))\chi(x,yz)$.
\end{enumerate}

There is a similar equivalence relation that can be placed on the set of factor systems, but here we must specifically require $\gamma$ to be invertible. 

This characterisation may be found in \cite{redei1952verallgemeinerung}. A discussion of the related notion of abstract kernels in a similar context may be found in \cite{martins2020classification}.

\section{Weakly Schreier extensions}
While Schreier extensions are the correct context to emulate the group theoretic arguments in the monoid setting, there are many natural examples of monoid extensions which are not Schreier. 

For instance in \cite{leech1982extending}, Leech considers extensions $\normalext{N}{k}{G}{e}{H}$ in which $N$ is a group and $gN = Ng$ for all $g \in G$. Suppose $\mathbb{Z}_\infty = \mathbb{Z}\cup \{\infty\}$ with $n + \infty = \infty = \infty + n$ and $2 = \{\top,\bot\}$ under meet. The extension $\normalext{\mathbb{Z}}{k}{\mathbb{Z}_\infty}{e}{2}$, with $k(n) = n$, $e(n) = \top$, $e(\infty) = \bot$ constitutes a normal Leech extension which is not Schreier. The failure is due to the lack of a (strict) generator in $e\inv(\bot)$. The only choice is $\infty$ but we see that $n + \infty = \infty = n' + \infty$ for all $n,n' \in \mathbb{Z}$. It can be shown that these are examples of $\mathcal{H}$-coextensions though this will be again briefly touched on later.

Another example is given by Artin glueings of frames \cite{sga4vol1, wraith1974glueing}. Frames are algebraic structures which behave like the lattice of open sets for a topological space. Artin glueings of frames $H$ and $N$ are specified by finite-meet-preserving maps $f\colon H \to N$, where the Artin glueing $\Gl(f) = \{(n,h)\in N \times H: n \le f(h)\}$ has point wise meets and joins. As was shown in \cite{faul2019artin}, the diagram $\splitext{N}{k}{\Gl(f)}{e}{e_*}{H}$ with $k(n) = (n,1)$, $e(n,h) = h$ and $e_*(h) = (f(h),h)$ is a split extension of monoids (where meet is taken as the monoid multiplication). It is the case that $(n,h) = k(n) \wedge e_*(h)$, but in general this $n$ will almost never be unique.

There are a number of other examples of non-Schreier extensions, like Billhardt's $\lambda$-semidirect products \cite{billhardt1992wreath} of inverse monoids (as shown in \cite{faul2020lambda}) and Fulp and Stepp's monoid extensions \cite{fulp1971structure,stepp1971semigroup} (which are the central extensions of Leech's normal extensions of monoids \cite{leech1982extending}).  

What these examples all share in common (excluding our brief mention of $\mathcal{H}$-coextensions) is that in each fibre $e\inv(h)$ there exists a `weak' generator $u_h$. We define this formally below with the notion of a weakly Schreier extension.

\begin{definition}
An extension $\normalext{N}{k}{G}{e}{H}$ is weakly Schreier if in each fibre $e\inv(h)$ there exists an element $u_h$, called a (weak) generator such that for each element $g\in e\inv(h)$ there exists a (not necessarily unique) $n \in N$ such that $g = k(n)u_h$. 
\end{definition}

These were first considered in a relatively unknown paper of Fleischer \cite{fleischer1981monoid}. These extensions and their corresponding notions of weakly Schreier split extensions and special weakly Schreier extensions may characterised in a manner reminiscent of the group and Schreier case. 

\subsection{Weakly Schreier split extensions}
We begin with weakly Schreier split extensions.

\begin{definition}
A split extension $\splitext{N}{k}{G}{e}{s}{H}$ is weakly Schreier when for each $g \in G$ there exists a (not necessarily unique) $n \in N$ such that $g = k(n)se(g)$.
\end{definition}

Immediately we see that we will not get the usual bijection $f \colon N \times H \to G$ sending $(n,h)$ to $k(n)s(h)$. Instead this map is only guaranteed to be surjective. Thus instead, we may quotient $N \times H$ and then consider the induced bijection $\overline{f}\colon N \times H/{\sim} \to G$ where $(n_1,h_1) \sim (n_2,h_2)$ if and only if $k(n_1)s(h_1) = k(n_2)s(h_2)$ and where $f([n,h]) = k(n)s(h)$.

This equivalence relation will always satisfy the following four conditions.

\begin{enumerate}\setcounter{enumi}{-1}
    \item $(n,h) \sim (n',h')$ implies $h = h'$,
    \item $(n,1) \sim (n',1)$ implies $n = n'$,
    \item $(n,h) \sim (n',h)$ implies that $(xn,h) \sim (xn',h)$ for all $x \in N$ and
    \item $(n,h) \sim (n',h)$ implies that $(n,hx) \sim (n',hx)$ for all $x \in H$.
\end{enumerate}


Condition 0 implies that this equivalence relation on $N \times H$ may instead be viewed as an $H$-indexed equivalence relation on $N$. If we equip $H$ with the divisibility preorder, then condition 3 says that the assignment of equivalence relations to elements of $H$ is order preserving, condition 1 says that the assignment preserves bottom and condition 2 says that the equivalence relations are always right congruences.

We obtain the following definition.

\begin{definition}
We call an $H$-indexed equivalence relation on $N$ an \emph{$H$-relaxation} of $N$ if the following conditions hold.

\begin{enumerate}
    \item $n_1 \sim^1 n_2$ implies $n_1 = n_2$,
    \item $n_1 \sim^h n_2$ implies $xn_1 \sim^h xn_2$,
    \item $n_1 \sim^{h_1} n_2$ implies $n_1 \sim^{h_1h_2} n_2$.
\end{enumerate}
\end{definition}

From this new perspective we see that $G$ is in bijection with the disjoint union $\bigsqcup_{h\in H}N/{\sim^h}$.

The multiplication may be characterised by something very similar to an action. As with Schreier split extensions we consider a function $\alpha \colon H \times N \to N$ satisfying that $k\alpha(h,n)s(h) = s(h)k(n)$. In this setting this $\alpha$ is not necessarily unique, however any other $\alpha'$ may be related to $\alpha$ via $\alpha(h,n) \sim^h \alpha'(h,n)$, for all $h \in H$ and $n \in N$.

It can be shown that such maps $\alpha$ always satisfy six conditions which we now collect into a definition.

\begin{definition}
Let $E$ be an $H$-relaxation of $N$, a function $\alpha\colon H \times N \to N$ is a compatible action if the following conditions hold.
\begin{enumerate}
    \item $n_1 \sim^h n_2$ implies $n_1\alpha(h,n) \sim^h n_2\alpha(h,n)$ for all $n \in N$,
    \item $n_1 \sim^{h_2} n_2$ implies $\alpha(h_1,n_1) \sim^{h_1h_2} \alpha(h_1,n_2)$ for all $h_1 \in H$,
    \item $\alpha(h,n_1n_2) \sim^h \alpha(h,n_1)\cdot\alpha(h,n_2)$,
    \item $\alpha(h_1h_2,n) \sim^{h_1h_2} \alpha(h_1,\alpha(h_2,n))$,
    \item $\alpha(h,1) \sim^h 1$,
    \item $\alpha(1,n) \sim^1 n$.
\end{enumerate}
We quotient the set of compatible actions via $\alpha \sim \alpha'$ if and only if $\alpha(h,n) \sim^h \alpha'(h,n)$ for all $h \in H$ and $n \in N$.

We call the pair $(E,[\alpha])$ a \emph{relaxed action} of $H$ on $N$.
\end{definition}

In particular notice that conditions 3-6 are the usual conditions of an action, but which now only hold up to $H$-equivalence. 

Given a compatible action $\alpha$ we may now equip $\bigsqcup_{h\in H}N/{\sim^h}$ with a multiplication making it isomorphic to $G$. This is given by 
\[([n_1],h_1)([n_2],h_2) = ([n_1\alpha(h_1,n_2)],h_1h_2).
\] 
It is not hard to see that for any other $\alpha' \in [\alpha]$, we get the same multiplication.

Call the resulting monoid $N \rtimes_{E,\alpha} H$ the \emph{relaxed semidirect product} associated to the relaxed action $(E,\alpha)$. It may naturally be considered a part of a weakly Schreier split extension $\splitext{N}{k'}{N \rtimes_{E,\alpha} H}{e'}{s'}{H}$ where $k(n) = ([n],1)$, $e([n],h) = h$ and $s(h) = ([1],h)$. This extension will be isomorphic to $\splitext{N}{k}{G}{e}{s}{H}$.

In fact, one may start with an arbitrary relaxed action and construct the associated weakly Schreier split extension. These constitute a full characterisation of weakly Schreier split extensions.

This characterisation was first provided in \cite{kohler1979quasi} in the context of characterising quasi-decompositions of monoids. This result was an extension of the work in \cite{schmidt1977quasi} in which the commutative case was considered. It also generalised Krishnan's quasi-resolution found in \cite{krishnan1979quasi}.

This work was unknown to the author who rediscovered this characterisation in \cite{faul2019characterisation}. The two papers have different focuses with the former being concerned primarily with structural results and the latter more concerned with the extension perspective. 

From the perspective of extensions it is worth mentioning that the split short five lemma does not hold for weakly Schreier split extensions. The full characterisation of the morphisms between weakly Schreier split extensions yields a preorder category. In particular this implies that the set of relaxed actions of $H$ on $N$ have a non-trivial order structure.

\subsection{Weakly Schreier extensions with abelian group kernel}
Special weakly Schreier extensions are those weakly Schreier extensions which have group kernel. An equivalent definition may be given as follows.

\begin{definition}
An extension $\normalext{N}{k}{G}{e}{H}$ is special weakly Schreier (or cosetal) if whenever $e(g_1) = e(g_2)$ then there exists a (not necessarily unique) $n \in N$ such that $g_1 = k(n)g_2$.
\end{definition}

When $N$ is abelian we may provide a simple characterisation in terms of factor sets completely analogous to the group case.

First, to each such extension $\normalext{N}{k}{G}{e}{H}$ we may associate a unique relaxed action of $H$ on $N$. Take any set-theoretic splitting $s$ of $e$ that preserves the unit and define an $H$-indexed equivalence relation $E_s$ given by
\[
n_1 \sim^h n_2 \iff k(n_1)s(h) = k(n_2)s(h). 
\]

It is not too hard to see that all of the conditions of an $H$-relaxation of $N$ are met. If $s'$ is some other splitting notice that $es(h) = es'(h)$ and hence by assumption there exists an element $t(h) \in N$ such that $s(h) = t(h)s'(h)$. Thus if $k(n_1)s'(h) = k(n_2)s'(h)$, multiplying both sides by $kt(h)$ yields $k(n_1)s(h) = k(n_2)s(h)$. Using a symmetric argument we may then conclude that $E_s = E_{s'}$. Hence we may safely denote this relaxation by $E$.

To find the compatible action observe that $e(s(h)) = e(s(h)k(n))$ and hence there exists an element $\alpha(h,n) \in N$ such that $k\alpha(h,n)s(h) = s(h)k(n)$. It is easily shown that $\alpha$ satisfies the axioms of a compatible action and moreover does not depend on the choice of splitting $s$.

Hence we may partition the set of isomorphism classes of special weakly Schreier extensions according to their associated relaxed action. We may then characterise these extensions with a relaxed notion of a factor set. 

Given a relaxed action $(E,[\alpha])$ a relaxed factor set is a function $g\colon H \times H \to N$ satisfying that

\begin{enumerate}
    \item $\chi(1,h) \sim^h 1 \sim^h \chi(h,1)$,
    \item $\chi(x,y)\chi(xy,z) \sim^{xyz} \alpha(x,\chi(y,z))\chi(x,yz)$.
\end{enumerate}

Given such data we may construct the relaxed crossed product $N \rtimes_{E,\alpha}^{\chi} H$ with underlying set $\bigsqcup_{h \in H}N/{\sim^h}$ and multiplication given by 
\[
([n_1],h_2)([n_2],h_2) = ([n_1\alpha(h_1,n_2)\chi(h_1,h_2)],h_1h_2).
\]

This may be completed into an extension $\normalext{N}{k}{N \rtimes_{E,\alpha}^\chi H}{e}{H}$ with $k(n) = ([n],1)$ and $e([n],h) = h$. 
Relaxed factor sets give the same multiplication whenever $\chi(h_1,h_2) \sim^{h_1h_2} \chi'(h_1,h_2)$ for all $h_1,h_2 \in H$. We may indentify such relaxed factor sets.
As before we may consider $Z_2(H,N,E,[\alpha])$ of all (equivalence classes of) factor sets associated to $(E,[\alpha])$. This set has a natural abelian group structure. Many factor sets give the same extensions and so a full characterisation is only provided after quotienting by an appropriate notion of inner factor set. These factor sets are defined as in the group and strict Schreier case. 

The result is $\mathcal{H}^2(H,N,E,[\alpha])$, the second cohomology group. It imbues a Baer sum on extensions with the same action.

This characterisation may be found in full (under the name cosetal extensions) in \cite{faul2020baer}. These ideas were then further extended in \cite{faul2021quotients} which accounted for functorial aspect arising from the non-trivial order structure on relaxed actions.

It is easily verified that Leech's normal extensions are all special weakly Schreier. The converse is not true. If we consider the action $\alpha\colon 2 \times \mathbb{Z} \to \mathbb{Z}$ such that $\alpha(\top,n) = n$ and $\alpha(\bot,n) = 0$ for all $n \in \mathbb{Z}$. The associated Schreier split extension is special Schreier (and consequently special weakly Schreier), but it is easy to verify that $(0,\bot)\mathbb{Z} \ne \mathbb{Z}(0,\bot)$.

\subsection{Weakly Schreier extensions}
As mentioned before, general weakly Schreier extensions were first considered by Fleischer in \cite{fleischer1981monoid}, in which he provided a full characterisation. While his proof makes much use of concepts very related to the notion of relaxed actions they are never treated as a concept themselves. In this section we will give a presentation of his characterisation completely analogous to the group case only with relaxed actions replacing classical actions.

In the group setting we begin with an extension $\normalext{N}{k}{G}{e}{H}$ and induce an action of $G$ on $N$. We then compose this action with a set-theoretic splitting $s$ of $e$ which selects only generators and then find the related factor sets.

If $\normalext{N}{k}{G}{e}{H}$ is a weakly Schreier extension of monoids we will not in general be able to induce a relaxed action of $G$ on $N$. However, there will exist a relaxed action of the \emph{right-normaliser} of $N$ on $N$.

\begin{definition}
Let $N$ be a submonoid of $G$. Then the right-normaliser $\mathcal{N}_G(N)$ is the submonoid given by $\{g \in G : gN \subseteq Ng\}$.
\end{definition}

Notice that this is the set of magnifying elements of $N$ as defined in \cite{migliorini1971some,gutan1997semigroups,tolo1969factorizable}.

We may now describe the relaxed action induced by a weakly Schreier extension. Let $\normalext{N}{k}{G}{e}{H}$ be a weakly Schreier extension of monoids. We may define an $\mathcal{N}_G(N)$-indexed equivalence relation $E$ on $N$ via $n_1 \sim^g n_2$ if and only if $k(n_1)g = k(n_2)g$. It may be verified easily that this is an $\mathcal{N}_G(N)$-relaxation of $N$.

Now if $g \in \mathcal{N}_G(N)$ then $g \cdot k(n) \in gN \subseteq Ng$. Hence, there exists an element $\alpha(g,n) \in N$ such that $k\alpha(g,n)g = g \cdot k(n)$. Let us verify that this is a compatible action.

First we suppose that $n_1 \sim^g n_2$ and consider the following.
\begin{align*}
k(n_1)k\alpha(g,n)g 
&= k(n_1) \cdot g \cdot k(n) \\ 
&= k(n_2) \cdot g \cdot k(n) \\
&= k(n_2)k\alpha(g,n)g
\end{align*}
This allows us to conclude that $n_1 \sim^g n_2$ implies $n_1\alpha(g,n) \sim^g n_2\alpha(g,n)$.

Next we suppose that $n_1 \sim^{g_2} n_2$ and consider the following.
\begin{align*}
k\alpha(g_1,n_1)g_1g_2 
&= g_1 \cdot k(n_1)g_2 \\ 
&= g_1 \cdot k(n_2)g_2 \\ 
&= k\alpha(g_1,n_2)g_1g_2
\end{align*}
Thus we may conclude that $n_1 \sim^{g_2} n_2$ implies $\alpha(g_1,n_1) \sim^{g_1g_2} \alpha(g_1,n_2)$.

Next we consider the following calculation.
\begin{align*}
k\alpha(g,n_1n_2)g 
&= g \cdot k(n_1)k(n_2) \\ 
&= k\alpha(g,n_1) \cdot g \cdot k(n_2) \\ 
&= k\alpha(g,n_1)k\alpha(g,n_2)g
\end{align*}
This then gives that $\alpha(g,n_1n_2) \sim^g \alpha(g,n_1)\alpha(g,n_2)$.

The fourth condition is given by the following calculation.
\begin{align*}
k\alpha(g_1g_2,n)g_1g_2 
&= g_1g_2 \cdot k(n) \\ 
&= g_1 \cdot k\alpha(g_2,n)g_2 \\ 
&= k\alpha(g_1,\alpha(g_2,n))g_1g_2
\end{align*}
This of course allows us to conclude that $\alpha(g_1g_2,n) \sim^{g_1g_2} \alpha(g_1,\alpha(g_2,n))$.

The final two conditions are easily deduced. We have $k\alpha(g,1)g = g \cdot k(1) = g$ and $k\alpha(1,n) \cdot 1 = k(n)$ which imply that $\alpha(g,1) \sim^g 1$ and $\alpha(1,n) \sim^1 n$ respectively.

Thus, $(E,[\alpha])$ is a relaxed action of $\mathcal{N}_G(N)$ on $N$. We would like to now compose with a set-theoretic splitting $s$ of $e$ which selects generators. We must verify that the generators $u_h$ lie in the normaliser. This is easily seen to be the case as $u_hk(n) \in e\inv(h)$ and hence by definition there exists an $x \in N$ such that $u_hk(n) = k(x)u_h$.

Now suppose $s$ splits $e$ and each $s(h)$ is a generator. Since $s(h_1)s(h_2) \in e\inv(h_1h_2)$ we know that there exists an element $\chi(h_1,h_2) \in N$ such that $k\chi(h_1,h_2)s(h_1h_2) = s(h_1)s(h_2)$. We may assume that $\chi(h,1) = 1 = \chi (1,h)$. We can then `compose' the relaxed action $(E,[\alpha])$ with $s$ to yield $(E_s,\alpha_s)$. Where $n_1 \sim^h n_2$ if and only if $k(n_1)s(h) = k(n_2)s(h)$ and $\alpha_s(h,n) = \alpha(s(h),n)$.

The pair $(E_s,\alpha_s)$ is no longer a relaxed action as all identities involving $\sim^{h_1h_2}$ do not in general hold anymore. However, modified versions of these identities involving $\chi$ do hold.

First, if we have that $n_1 \sim^{h_1} n_2$ then we may consider the following calculation.
\begin{align*}
k(n_1)k\chi(h_1,h_2)s(h_1h_2) 
&= k(n_1)s(h_1)s(h_2) \\
&= k(n_2)s(h_1)s(h_2) \\
&= k(n_2)k\chi(h_1,h_2)s(h_1h_2)
\end{align*}
Hence we get that $n_1 \sim^{h_1} n_2$ implies $n_1\chi(h_1,h_2) \sim^{h_1h_2} n_2\chi(h_1,h_2)$.

Next we have that if $n_1 \sim^{h_2} n_2$ the following holds true.
\begin{align*}
k\alpha_s(h_1,n_1)k\chi(h_1,h_2)s(h_1h_2) 
&= k\alpha_s(h,n_1)s(h_1)s(h_2) \\
&= s(h_1)k(n_1)s(h_2) \\
&= s(h_1)k(n_2)s(h_2) \\
&= k\alpha_s(h_1,n_2)k\chi(h_1,h_2)s(h_1h_2)
\end{align*}
Thus we have that $\alpha(h_1,n_1)\chi(h_1,h_2) \sim^{h_1h_2} \alpha(h_1,n_2)\chi(h_1,h_2)$.

Finally we have the following.
\begin{align*}
k\chi(h_1,h_2)k\alpha_s(h_1h_2,n)s(h_1h_2) 
&= k\chi(h_1,h_2)s(h_1h_2)k(n) \\ 
&= s(h_1)s(h_2)k(n) \\
&= s(h_1)k\alpha_s(h_2,n)s(h_2) \\
&= k\alpha_s(h_1,\alpha_s(h_2,n))s(h_1)s(h_2) \\
&= k\alpha_s(h_1,\alpha_s(h_2,n))k\chi(h_1,h_2)s(h_1h_2)
\end{align*}
Hence we have that $\chi(h_1,h_2)\alpha_s(h_1h_2,n) \sim^{h_1h_2} \alpha_s(h_1,\alpha_s(h_2,n))\chi(h_1,h_2)$.

We summarise the above in the following definition.

\begin{definition}\label{deffac}
We call $(E_s,\alpha_s,\chi)$ a factor system when it satisfies the following properties.
\begin{enumerate}
    \item $n_1 \sim^ 1 n_2$ implies $n_1 = n_2$,
    \item $n_1 \sim^h n_2$ implies $xn_1 \sim^h xn_2$,
    \item $n_1 \sim^{h_1} n_2$ implies $n_1\chi(h_1,h_2) \sim^{h_1h_2} n_2\chi(h_1,h_2)$,
    \item $n_1 \sim^h n_2$ implies $n_1\alpha(h,n) \sim^h n_2\alpha(h,n)$ for all $n \in N$,
    \item $n_1 \sim^{h_2} n_2$ implies $\alpha(h_1,n_1)\chi(h_1,h_2) \sim^{h_1h_2} \alpha(h_1,n_2)\chi(h_1,h_2)$ for all $h_1 \in H$,
    \item $\alpha(h,n_1n_2) \sim^h \alpha(h,n_1)\cdot\alpha(h,n_2)$,
    \item $\chi(h_1,h_2)\alpha(h_1h_2,n) \sim^{h_1h_2} \alpha(h_1,\alpha(h_2,n))\chi(h_1,h_2)$,
    \item $\alpha(h,1) \sim^h 1$,
    \item $\alpha(1,n) \sim^1 n$,
    \item $\chi(1,h) \sim^h 1 \sim^h \chi(h,1)$,
    \item $\chi(x,y)\chi(xy,z) \sim^{xyz} \alpha(x,\chi(y,z))\chi(x,yz)$.
\end{enumerate}
\end{definition}

As is to be expected, a crossed product may be defined from the above data and the resulting extension will be isomorphic to the original. To prove this we introduce the following notation:

\begin{enumerate}
    \item $x \ast ([n],h) = ([xn],h)$,
    \item $([n],h) \ast y = ([n\chi(h,y)],hy)$.
\end{enumerate}

The first operation is well-defined via condition (2) above and the second operation via condition 3. 

Now we consider $N \rtimes_{E,\alpha}^\chi H$ with underlying set $\bigsqcup_{h \in H} N/{\sim^h}$ and multiplication given by 
\[
([n_1],h_1)([n_2],h_2) = ([n_1\alpha(h_1,n_2)\chi(h_1,h_2)],h_1h_2).
\]
The unit is $([1],1)$.

\begin{proposition}\label{propfac}
If $(E,\alpha,\chi)$ is a factor system, then $N \rtimes_{E,\alpha}^\chi H$ is a monoid.
\end{proposition}

\begin{proof}
We first prove that the multiplication is well-defined. Suppose that $([n_1],h_1) = ([n_2],h_1)$ and $([x_1],h_2) = ([x_2],h_2)$. Now simply consider the following calculation.

\begin{align*}
    ([n_1],h_1)([x_1],h_2) 
    &= ([n_1\alpha(h_1,x_1)\chi(h_1,h_2)],h_1h_2) \\
    &= n_1 \ast ([\alpha(h_1,x_1)\chi(h_1,h_2)],h_1h_2) \\
    &= n_1 \ast ([\alpha(h_1,x_2)\chi(h_1,h_2)],h_1h_2) \\
    &= ([n_1\alpha(h_1,x_2)],h_1) \ast h_2 \\
    &= ([n_2\alpha(h_1,x_2)],h_1) \ast h_2 \\
    &= ([n_2\alpha(h_1,x_2)\chi(h_1,h_2)],h_1h_2) \\
    &= ([n_2],h_1)([x_2],h_2)
\end{align*}

It is clear that $([1],1)$ is the identity and so we must just show that the multiplication is associative. Consider the following calculation.

\begin{align*}
    (([n_1],h_1)([n_2],h_2))([n_3],h_3) 
    &= ([n_1\alpha(h_1,n_2)\chi(h_1,h_2)],h_1h_2)([n_3],h_3) \\
    &= ([n_1\alpha(h_1,n_2)\chi(h_1,h_2)\alpha(h_1h_2,n_3)\chi(h_1h_2,h_3)],h_1h_2h_3) \\
    &= n_1\alpha(h_1,n_2) \ast ([\chi(h_1,h_2)\alpha(h_1h_2,n_3)],h_1h_2) \ast h_3 \\
    &= n_1\alpha(h_1,n_2) \ast ([\alpha(h_1,\alpha(h_2,n_3))\chi(h_1,h_2)],h_1h_2) \ast h_3 \\
    &= n_1\alpha(h_1,n_2)\alpha(h_1,\alpha(h_2,n_3)) \ast ([\chi(h_1,h_2)\chi(h_1h_2,h_3)],h_1h_2h_3) \\
    &= n_1\alpha(h_1,n_2)\alpha(h_1,\alpha(h_2,n_3)) \ast ([\alpha(h_1,\chi(h_2,h_3))\chi(h_1,h_2h_3)],h_1h_2h_3)
\end{align*}

We may then compare this with the following.

\begin{align*}
    ([n_1],h_1)(([n_2,h_2)([n_3],h_3)) 
    &= ([n_1],h_1)([n_2\alpha(h_2,n_3)\chi(h_2,h_3)],h_2h_3) \\
    &= ([n_1\alpha(h_1,n_2\alpha(h_2,n_3)\chi(h_2,h_3))\chi(h_1,h_2h_3),h_1h_2h_3) \\
    &= n_1 \ast ([\alpha(h_1,n_2\alpha(h_2,n_3)\chi(h_2,h_3))],h_1) \ast h_2h_3 \\
    &= n_1 \ast ([\alpha(h_1,n_2)\alpha(h_1,\alpha(h_2,n_3))\alpha(h_1,\chi(h_2,h_3)))],h_1) \ast h_2h_3 \\
    &= n_1\alpha(h_1,n_2)\alpha(h_1,\alpha(h_2,n_3)) \ast ([\alpha(h_1,\chi(h_2,h_3))\chi(h_1,h_2h_3)],h_1h_2h_3)
\end{align*}

It is clear that these two expressions are equal and so indeed this is a monoid as desired.
\end{proof}

If we start with a weakly Schreier extension $\normalext{N}{k}{G}{e}{H}$ and extract an associated factor system $(E,\alpha,\chi)$ corresponding to the splitting $s$. Then the map $f \colon N \rtimes_{E,\alpha}^\chi H \to G$, sending $([n],h)$ to $k(n)s(h)$ is an isomorphism.

By construction it is a bijection of sets and so we must simply verify that it is a homomorphism. It is clear that $f([1],1) = 1$ and so we consider only the following calculation.

\begin{align*}
    f([n_1],h_1)f([n_2],h_2) 
    &= k(n_1)s(h_1)k(n_2)s(h_2) \\
    &= k(n_1)k\alpha(h_1,n_2)s(h_1)s(h_2) \\
    &= k(n_1)k\alpha(h_1,n_2)k\chi(h_1,h_2)s(h_1h_2) \\
    &= f([n_1\alpha(h_1,n_2)\chi(h_1,h_2)],h_1h_2) \\
    &= f(([n_1],h_1)([n_2],h_2))
\end{align*}

Moreover, we may construct the extension $\normalext{N}{k}{N \rtimes_{E,\alpha}^\chi}{e'}{H}$ with $k(n) = ([n],1)$ and $e([n],h) = h$. It is not hard to see that $f$ is then an isomorphism of extensions.

All that remains is some condition to identify when two factor systems give isomorphic extensions. Suppose that in the following, $f$ is an isomorphism of extensions.

\begin{center}
  \begin{tikzpicture}[node distance=2.0cm, auto]
    \node (A) {$N$};
    \node (B) [right=1.2cm of A] {$N \rtimes_{E,\alpha}^\chi H$};
    \node (C) [right=1.2cm of B] {$H$};
    \node (D) [below of=A] {$N$};
    \node (E) [below of=B] {$N \rtimes_{E',\alpha'}^{\chi'} H$};
    \node (F) [below of=C] {$H$};
    \draw[normalTail->] (A) to node {$k$} (B);
    \draw[-normalHead] (B) to node {$e$} (C);
    \draw[normalTail->] (D) to node {$k'$} (E);
    \draw[-normalHead] (E) to node {$e'$} (F);
    \draw[->] (B) to node {$f$} (E);
    \draw[double equal sign distance] (A) to (D);
    \draw[double equal sign distance] (C) to (F);
  \end{tikzpicture}
\end{center}

Just as in the group case we have $f([n],h) = f([n],1)f([1],h) = ([n],1)f([1],h)$. Since $f$ must preserve the second component we may write $f([1],h) = ([\gamma(h)],h)$ for some set-theoretic function $\gamma \colon H \to N$ preserving 1. Then we have that $f([n],h) = ([n\gamma(h)],h)$.

We would then like to determine which functions $\gamma \colon H \to N$ may be used in this way to construct an isomorphism between two crossed products. The result will be a homomorphism if whenever $n_1 \sim^{h_1h_2} n_2$ we have
\[\gamma(h_1h_2)n_1\alpha'(h_1,n_2)\chi'(h_1,h_2) \sim^{h_1h_2} \gamma(h_1)n_1\alpha(h_1,\gamma(h_2)n_2)\chi(h_1,h_2).
\]
However, as was shown in \cite{faul2021quotients} morphisms of extensions are not guaranteed to be isomorphisms. 

If we construct $f$ from $\gamma$ and the result has an inverse $f\inv$ then there is necessarily some $\lambda \colon H \to N$ such that $([\gamma(h)\lambda(h)],h) = ([1],h)$ and $([\lambda(h)\gamma(h)],1)$. Hence $\gamma(h)$ must be right invertible relative to $h$ in $E$ and left invertible relative to $h$ in $E'$. Conversely it is not hard to see that these conditions give that $f$ is injective and surjective respectively and hence an isomorphism.
This suffices to characterise all such $\gamma$. Hence we now have a complete characterisation of weakly Schreier extensions.

We summarise these results in the following theorem.

\begin{theorem}
The crossed product construction (described in \cref{propfac}) gives a bijective correspondence between weakly Schreier extensions and the collection of factor systems (as in \cref{deffac}) modulo the equivalence relation given by $(E,\alpha,\chi) \sim (E',\alpha',\chi')$ whenever there exists a map $\gamma$ that is right invertible relative to $h$ in $E$ and left invertible relative to $h$ in $H'$ satisfying that $n_1 \sim^{h_1h_2} n_2$ implies
\[\gamma(h_1h_2)n_1\alpha'(h_1,n_2)\chi'(h_1,h_2) \sim^{h_1h_2} \gamma(h_1)n_1\alpha(h_1,\gamma(h_2)n_2)\chi(h_1,h_2).
\]
\end{theorem}

We end by discussing related work which does not quite fit into this paradigm.

In \cite{fleischer1981monoid}, it is pointed out that Grillet's left coset extensions as well as the equivalent $\mathcal{H}$-coextensions of Leech are in general incomparable to weakly Schreier extensions.

In \cite{martins2020semi}, semi-biproducts of monoids are considered which are more general than weakly Schreier extensions, though which contain extra data in the form of set theoretic splittings for the kernel and cokernel. They are characterized in terms of pseudo actions, comprising three parts. The first two are a preaction and a correction system which together appear to play the role of a relaxed action, except that instead of considering equivalence classes, a representative is selected for each. The final part is a factor system. Further work is required to more clearly understand the connections between this work and the extensions described here.

\bibliographystyle{abbrv}
\bibliography{bibliography}

\end{document}